\documentclass[12pt]{article}
\usepackage[centertags]{amsmath}
\usepackage{amsfonts}
\usepackage{amssymb}
\usepackage{amsthm}
\usepackage{newlfont}
\usepackage[latin1]{inputenc}
\newtheorem{fed}{Definition}[section]
\newtheorem{teo}[fed]{Theorem}
\newtheorem{lem}[fed]{Lemma}

\newtheorem{rem}[fed]{Remark}
\newtheorem{prop}[fed]{Proposition}
\newtheorem{defn}[fed]{Definition}

 \DeclareMathOperator{\RE}{Re}
 
\begin{document}
\title{Asymptotic behavior of global solutions of the $u_t=\Delta u + u^{p}$  \footnote{ Key
words: Semilinear heat equation, asymptotic behavior, entropy
method, entropy functional, entropy production functional.} }

\author{ Oscar A. Barraza \quad Laura B. Langoni}

\maketitle \vskip-2cm
\ \\
\begin{center}
{\it Departamento de Matem\'atica, Facultad de Ciencias Exactas,
 Universidad Nacional de La Plata, C.C. 172, (1900) La Plata, Argentina.
 e-mail: oscar@mate.unlp.edu.ar}
\end{center}
\vskip3cm




\begin{abstract}

\noindent We study the asymptotic behavior of nonnegative
solutions of the semilinear parabolic problem
\[
\left\{
\begin{array}{l}
u_t=\Delta u + u^{p}, \;\; x\in\mathbb{R}^{N},\;\;t>0\\
u(0)=u_{0}, \;\; x\in\mathbb{R}^{N},\;\;t=0.
\end{array}
\right.
\]
It is known that  the nonnegative solution $u(t)$ of this problem
blows up in finite time for $1<p\leq 1+ 2/N$. Moreover, if $p> 1+
2/N$ and the norm of $u_{0}$ is small enough, the problem admits
global solution. In this work, we use the entropy method to obtain
the decay rate of the global solution $u(t)$.

\end{abstract}

\vfill \eject

\section{Introduction}

We considerer the semilinear parabolic problem
\begin{equation}\label{ec}
\left\{
\begin{array}{l}
u_t=\Delta u + u^{p}, \;\; x\in\mathbb{R}^{N},\;\;t>0\\
u(0)=u_{0}, \;\; x\in\mathbb{R}^{N},\;\;t=0,
\end{array}
\right.
\end{equation}
where $p>1$ and $u_{0}$ is nonnegative and nontrivial. The
interest in this article is to study the asymptotic behavior of
global in time solutions in order to obtain their decay rate.
Although this matter has been treated for other authors, we do it
using the entropy method. This method has been successfully
applied, for instead,  by J.A. Carrillo and G. Toscani
\cite{Carrillo-Toscani} for the asymptotic behavior of global
solutions of some Fokker-Planck type equations.

  In \cite{Fujita}, Fujita considered the evolution problem
(\ref{ec}) and proved the existence of a critical exponent $
p^{*}=1+\frac{2}{N}$, which is called the Fujita's exponent. This
exponent satisfies
\begin{itemize}
\item for $p>p^{*}$, if the norm of $u_{0}$ is small enough, there
exists a classical global in time positive solution of (\ref{ec})
which decays to zero when $t\rightarrow\infty$, in the other case
 the solution blows up in finite time;
 \item for
$1<p<p^{*}$ and any choice of the initial condition $u_{0}$, every
positive solution of  (\ref{ec}) blows up in finite time.
\end{itemize}
In the case $p=p^{*}$ Hayakawa \cite{Hayakawa} proved that every
positive solution of (\ref{ec}) has the same behavior as
$1<p<p^{*}$. This fact was proved by Kobayashi, Sirao and Tanaka
\cite{Kobayashi} too.

 The aim of this work is to obtain the decay rate
of the global solution $u(s)$ of the problem (\ref{ec}) when
$t\rightarrow\infty$. More specifically, we deduce that under
certain hypotheses, the $L^{2}$-distance of $u(s)$ decays with the
rate
\[\|u(t)\|_{L^{2}}\leq C\;
(t+1)^{-\frac{N}{4}}\;,\;\; t\geq t_{1}\;,\] for certain time
$t_{1}>0$, as well as the $L^{q}$-distance of $u(s)$
\[\|u(t)\|_{L^{q}}\sim
t^{-\frac{1}{p-1}-\frac{2}{q}\left(\frac{N}{4}-\frac{1}{p-1}\right)}.\]

 The organization of this work is the following. The section 2
contains the notation and some previous results which are
necessary for the next sections. In section 3 a short summary
about the entropy method steps is given.  In the section 4 we
obtain the exponential decay for the entropy production and for
the entropy.  The last section
 is devoted to deduce the asymptotic behavior of the mentioned global solution.

\section{Preliminaries}\label{prelim}

 The idea is to transform the equation in problem (\ref{ec}) in order to obtain significant
 information on the asymptotic behavior of the global solutions. The change of
 variables utilized for Kavian \cite{Kavian} and Kawanago \cite{Kawanago} is here employed. That
 is,
\begin{eqnarray}\label{cv}
v( y, s)&=& (t+1)^{\frac{1}{p-1}} \;u( x, t),\nonumber\\
x=(t+1)^{1/2}y   \;\; &\mbox{and}& \;\; t=e^{s}-1.
\end{eqnarray}
Then, $v( y, s)$ results to be a solution of the problem
\begin{equation}\label{ecv}
\left\{
\begin{array}{l}
v_s=\Delta v +\frac{y}{2}.\nabla v + \frac{v}{p-1}+ v^{p}, \;\; y\in\mathbb{R}^{N},\;\;s>0\\
v(y,0)=u_{0}, \;\; y\in\mathbb{R}^{N}.
\end{array}
\right.
\end{equation}
Let us observe that problem (\ref{ecv}) has the same initial
condition as problem (\ref{ec}). Moreover, we work with the
following spaces
\[
L^{r}_{\rho}=\left\{f\; /\int_{\mathbb{R}^{N}}|f|^{r}\rho \;dy
<\infty\right\}\; ,\]
\[H^{1}_{\rho}=\left\{f \in L^{2}_{\rho} / \; \nabla f  \in L^{2}_{\rho}\right\}\;,\;\; H^{2}_{\rho}=\left\{f \in
H^{1}_{\rho} /  \;\nabla f  \in H^{1}_{\rho}\right\}\;,\]
\noindent where $ \rho(y)=\exp(|y|^{2}/4)$ and $r\geq 1$ is a
constant. Related to these spaces are
\[
(f,g)_{L^{2}_{\rho}}= \int_{\mathbb{R}^{N}}f g\;\rho \; dy\;, \;
\; \|f\|_{L^{2}_{\rho}}=(f,f)^{\frac{1}{2}}\;,\]
\[(f,g)_{H^{1}_{\rho}}= (f,g)_{L^{2}_{\rho}}+(\nabla f,\nabla g)_{L^{2}_{\rho}}\;,\]
\[\|f\|_{H^{1}_{\rho}}=\left[\|f\|_{L^{2}_{\rho}}^{2}+\|\nabla
f\|_{L^{2}_{\rho}}^{2}\right]^{\frac{1}{2}}\;.\] \noindent Observe
that the equation in (\ref{ecv}) has to be written in the
 shape
\[v_{s}=-Lv+\frac{v}{p-1}+v^{p}\;\;,\]
where $L$ is the self-adjoint operator given by
\[Lv= -\triangle v -\frac{y}{2}.\nabla v,\; \textrm{defined over} \; D(L):=H^{2}_{\rho}.\]
\noindent We know that this operator satisfies
\begin{itemize} \item $\lambda_{1}=\frac{N}{2}$ is the least eigenvalue of
$L$. Then, the following inequality holds
\begin{equation}\label{desigL}
\frac{N}{2}\;\|v\|_{L^{2}_{\rho}}\leq \int_{\mathbb{R}^{N}}|\nabla
v|^{2}\rho \;dy\;,
\end{equation}
 \item the operator $L$ has compact inverse.
\end{itemize}
As well, it is known that
$u_{\infty}=C(N,p)\;|x|^{-\frac{2}{p-1}}$ with
$C(N,p)=\left[\frac{2}{p-1}\left(N-\frac{2p}{p-1}\right)\right]^{\frac{1}{p-1}}$
is a singular equilibrium of (\ref{ec}), that is, $u_{\infty}$ is
a solution of the Lane-Emden equation
\[ \triangle u + u^{p}=0, \;\;
x\in\mathbb{R}^{N}\;\;u>0,\;\;N\geq 3,\] which arises in
astrophysics and Riemannian geometry. It is well-known that this
fact is possible only for those values of $p$ that verify $p\geq
\frac{N}{N-2}$, since Gidas and Spruck \cite{Gidas} proved that
there are not stationary solutions in any other cases. In 1993,
Wang \cite{Wang} proved that if $N\geq 3$, $p>\frac{N}{N-2}$ and
\[0\leq u_{0}(x)\leq\lambda\; u_{\infty},\]
where $0<\lambda<1$, then (\ref{ec}) has a unique global classical
solution $u$ with  $0\leq u \leq \lambda u_{\infty}$. It also
satisfies that
\[u(x,t)\leq[(\lambda^{1-p}-1)(p-1)\;t]^{-\frac{1}{p-1}}.
\]
This inequality, in terms of the problem (\ref{ecv}), can  be
expressed as
\begin{equation}\label{deswangv}
v(y,s)\leq
\frac{1}{[(\lambda^{1-p}-1)(p-1)]^{\frac{1}{p-1}}}\left[\frac{e^{s}}{e^{s}-1}\right]^{\frac{1}{p-1}}.
\end{equation}
>From now on, we assume that the Wang's theorem hypotheses are
satisfied. Moreover, the framework will be the set of global
solutions $v$ of (\ref{ecv}) such that $v \in X$, with
\[X=\left\{f\in H^{1}_{\rho}\cap L^{\infty}/ \quad f\geq0
\;\; \textrm{and} \;
\;\lim_{s\rightarrow\infty}\int_{{\mathbb{R}^{N}}}|\nabla
f(s)|^{2}\rho\; dy=0\right\}.\]

\section{Entropy method}

We study the asymptotic behavior of the global in time solutions
of the problem (\ref{ecv}). For it, we use the already mentioned
entropy method. The essential application of this method will
consist in the following steps.
\begin{itemize}
\item Define a suitable \textbf{entropy functional
 $E(v(s))$} for the equation (\ref{ecv}) and study its properties.
  \item Compute the \textbf{entropy production}

\vspace{-0.8cm}
\[I(v(s))=\frac{d}{ds}E(v(s))\;.\]
\item Compute the derivative of entropy production and obtain a
differential equation of type

\vspace{-0.8cm}\[\frac{d}{ds}I(v(s)) = - \;C \;I(v(s))- R(s)\;,\]
for certain constant $C>0$ and certain function $R(s)$.
 \item Prove the properties of $R(s)$ that permit
to obtain an exponential decay of $I(v(s))$,

\vspace{-0.8cm}\[ I(v(s))\leq A \;e^{-Cs}\;.\] \item Obtain the
same decay rate for $E(v(s))$ from the previous items, more
specifically

\vspace{-0.8cm}\[ E(v(s))\leq B\; e^{-Cs}\;,\;\mbox{for}\;\;s\geq
s_{1}\;\textrm{and certain}\;s_{1}>0.\] \item Give a bound of
$\|v\|_{L^{2}_{\rho}}$ in terms of the entropy and entropy
production which permits to get conclusions on the decay of the
mentioned norm.
\end{itemize}

\medskip

The same entropy functional introduced by Kavian and Kawanago
(\cite{Kavian} and \cite{Kawanago} respectively) will be used in
the present article.

\begin{defn} For every $v\in H^{1}_{\rho}\cap L^{p+1}_{\rho}$
 the \textbf{ entropy functional} is defined by
\medskip
\[E(v)=\int_{\mathbb{R}^{N}}\left[\frac{1}{2}\;|\nabla v|^{2}-
\frac{1}{2(p-1)}\;v^{2}- \frac{1}{p+1}\;v^{p+1}\right]\rho\;
dy\;.\]
 \end{defn}

In order to obtain the decays announced above, some properties of
this functional are needed. These properties are summarized in the
next proposition. The first of them was proved in \cite{Kavian}.

\medskip

\begin{prop}\label{E}
Let $u_{0}\;\in H^{1}_{\rho}\cap L^{\infty}, u_{0}\geq 0$,
$E(u_{0})< \infty $ and $v=v(y,s)$ the global solution of
(\ref{ecv}), $v\in X$. Then
\begin{enumerate}
\item if there exists $s_{0}\geq 0$ such that $E(v(s_{0}))\leq 0$
and $v(s_{0})\neq 0$, $v$ blows up in finite time; \item
$\frac{d}{ds}E(v(s))=-I(v(s))$ where
$I(v(s))=\int_{\mathbb{R}^{N}} v^{2}_{s}\rho\;dy$; 
\item there exists $\displaystyle{\;
M:=\lim_{s\rightarrow\infty}E(v(s))}$ and, moreover, $M=0$.
\end{enumerate}
\end{prop}
\bigskip

\begin{defn}
 The functional $I(v(s))$ of the proposition \ref{E} is called
\textbf{entropy production}.
\end{defn}
\begin{proof}
 For the first property, the reader is refered to \cite{Kavian}. The second one is deduced by derivating
  $E(v(s))$ with respect to $s$, integrating by parts and
 keeping in mind that $v$ is solution of (\ref{ecv}).  In order to prove the third property,
 first observe that owing to the second one {\it 2} we have that
  $E(v(s))$ is non increasing. As $0 \leq E(v(s))\leq
E(v(0))< \infty $, the existence of the limit is warranted.
Therefore, we must only see that the limit is equal to zero. For
this issue, we observe that the first term in the expression of
$E(v(s))$ tends to zero when $s\rightarrow\infty$ since $v\in X$.
The second term of that expression goes to $0$ too when
$s\rightarrow\infty$ due to the inequality (\ref{desigL}). To see
 the behavior of the third term of the entropy functional we use the
inequality (\ref{deswangv}) and obtain
\begin{eqnarray*}
0  & \leq & \int_{\mathbb{R}^{N}} \frac{1}{p+1}v^{p+1}\rho\; dy \\
& \leq &
\left[\frac{e^{s}}{e^{s}-1}\right]\frac{1}{(\lambda^{1-p}-1)(p^{2}-1)}\int_{\mathbb{R}^{N}}v^{2}\rho
\; dy .
\end{eqnarray*}
The last expression can be bounded for large values of $s$ as
following
\begin{eqnarray*}
 \int_{\mathbb{R}^{N}} \frac{1}{p+1}v^{p+1}\rho\; dy  \leq
 C \frac{1}{(\lambda^{1-p}-1)(p^{2}-1)}\int_{\mathbb{R}^{N}}v^{2}\rho\; dy
,
\end{eqnarray*}
where $C$ is a positive constant. Taking into account the last
inequality, we get that the third term of $E(v(s))$ tends to zero
when $s\rightarrow\infty$.
\end{proof}
 Now, we want to prove the decay of the entropy production and, as a result, the decay of the
 entropy functional. It will be made in the next section.

\section{Decay of the entropy functional}

  The computation of $\frac{d I(v(s))}{ds}$ is needed to obtain the decay of the entropy
  production. We write this derivative in a convenient way
  using that $v$ is the solution of the equation (\ref{ecv}). That is,
 \begin{eqnarray}\label{igdI}
 \frac{d}{ds}I(v(s)) &=& \int_{\mathbb{R}^{N}}2\; v_{s}v_{ss}\rho\;
 dy\nonumber\\ &=& \frac{2}{p-1}\;I(v(s))-2\; (Lv_{s},v_{s})+2p
 \int_{\mathbb{R}^{N}}v^{p-1}v_{s}^{2}\rho \;dy\nonumber\\ &=& -2\gamma
 \;I(v(s))-2\;R(s),
\end{eqnarray}
where $\gamma$ is a positive constant and $R(s)$ is an appropriate
function. Both $\gamma$ and $R(s)$ are defined by
\begin{eqnarray}\label{R}
\gamma&=&\frac{N}{2}- \frac{1}{p-1}>0,\nonumber\\
R(s)&=&(Lv_{s},v_{s})_{L^{2}_{\rho}} - \frac{N}{2}
\;\|v_{s}\|^{2}_{L^{2}_{\rho}}- p
\int_{\mathbb{R}^{N}}v^{p-1}v^{2}_{s}\rho\; dy\;.
\end{eqnarray}
 We need to have more information about $R(s)$ that allows to obtain conclusions about the
  decay of the entropy production. This is the aim of the next three lemmas.
\begin{lem}\label{desiginteg1}
 Let $v\in
H^{2}_{\rho}$ and $\Omega\subset\mathbb{R}^{N}$ an open set. Then
\[\int_{\Omega}(vLv-\frac{N}{2}\;v^{2})\rho\;dy \geq 0\;.\]
\end{lem}
\begin{proof} We observe that, if  $\mathcal{X}_{\Omega}$ is the characteristic function of the set $\Omega$,
then
\begin{eqnarray*}
\int_{\Omega}v Lv \;\rho\;dy&=&\int_{\mathbb{R}^{N}}(vLv)
\;\mathcal{X}_{\Omega}\;\rho\;dy\\&=&\int_{\mathbb{R}^{N}}v\;\mathcal{X}_{\Omega}(Lv)
\;\mathcal{X}_{\Omega}\;\rho\;dy\;.\nonumber\end{eqnarray*}Owing
to $(Lv) \;\mathcal{X}_{\Omega}=L(v \;\mathcal{X}_{\Omega})$
almost everywhere, we have that
\begin{eqnarray}\label{desigcaract}
\int_{\Omega}v Lv
\;\rho\;dy&=&\int_{\mathbb{R}^{N}}v\;\mathcal{X}_{\Omega}L(v
\;\mathcal{X}_{\Omega})\;\rho\;dy\nonumber\\&\geq& \frac{N}{2}\;
\int_{\mathbb{R}^{N}}(v\;\mathcal{X}_{\Omega})^{2}\;\rho\;dy\\&=&
\frac{N}{2}\;
\int_{\mathbb{R}^{N}}v^{2}\;\mathcal{X}_{\Omega}\;\rho\;dy
\nonumber\\&=&\frac{N}{2}\; \int_{\Omega}v^{2}\;\rho\;dy,\nonumber
\end{eqnarray}
where the inequality in (\ref{desigcaract}) is a direct
consequence of (\ref{desigL}). From the last computation the
lemma's statement is established.
\end{proof}
\begin{lem}\label{Gn}
Let $f\in L^{1}(\mathbb{R}^{N})$ such that $\int_{\Omega}f\;dy\geq
0$ for every open set $\Omega\subset\mathbb{R}^{N}$. Then,
 $f\geq 0$ almost everywhere in $\mathbb{R}^{N}$.
\end{lem}
\begin{proof}
We suppose that $f$ isn't nonnegative a.e., then there exists a
measurable set $\Omega$ such that $m(\Omega)>0$ and $f<0$ in
$\Omega$. Here, $m(\Omega)$ denotes the Lebesgue measure in
$\mathbb{R}^{N}$ of the set $\Omega$. Then,
\[\alpha=\int_{\Omega}f\;dy<0.\]For each $n\in\mathbb{N}$, there exists
an open set $G_{n}\subset \mathbb{R}^{N}$ such that $\Omega\subset
G_{n}$ and $m(G_{n}-\Omega)< \frac{1}{n}$. We can choose those
open sets $G_{n}$ in such a way that the sequence $(G_{n})_{n}$ is
increasing with the inclusion. Then, we observe that
\begin{eqnarray}\label{desigalfa}
\int_{G_{n}}f\;dy&=&\alpha +
\int_{G_{n}-\Omega}f\;dy\nonumber\\&=&\alpha +
\int_{\mathbb{R}^{N}} \varphi_{n}\;dy
\end{eqnarray}
where $\varphi_{n}=\mathcal{X}_{(G_{n}-\Omega)} f $. It is quite
clear that the functions $\varphi_{n}$ are integrable functions,
they satisfy $|\varphi_{n}|\leq |f|$ and, moreover, they verify
that $\varphi_{n}\rightarrow 0$ when $n\rightarrow \infty$ a.e.
Then, owing to dominated convergence theorem, it results that
\begin{eqnarray*}\lim_{n\rightarrow\infty} \int_{\mathbb{R}^{N}} \varphi_{n}\;dy =
0.
\end{eqnarray*}
Thanks to the last equality we deduce that there exists a natural
$N_{0}$, such that for every $n\geq N_{0}$

\vspace{-1cm}
\begin{eqnarray}
\left|\int_{\mathbb{R}^{N}} \varphi_{n}\;dy\;\right|
<\frac{|\alpha|}{2}\;.
\end{eqnarray}
Using this inequality in (\ref{desigalfa}), we see that it
verifies
\begin{eqnarray*}
\int_{G_{n}}f\;dy &\leq &\alpha +
\left|\int_{\mathbb{R}^{N}}\varphi_{n}\;dy\right|\\&< &  \alpha +
\frac{|\alpha|}{2} =\frac{\alpha}{2}\;,
\end{eqnarray*}
for every $n\geq N_{0}$. That is, the integral of the function $f$
on the open sets $G_{n}$, for $n\geq N_{0}$, is negative. This
fact contradicts the lemma's hypothesis.
\end{proof}

The next lemma provides a bound of $R(s)$ in terms of a new
nonnegative function with exponential decay.
\begin{lem}\label{propk}
     Under the same hypotheses in the Wang's theorem and if, moreover,
      $\lambda<\left[\frac{3p-1}{\gamma(p-1)^{2}}+1\right]^{\frac{1}{1-p}}$, then the function
      $R(s)$ defined by (\ref{R}) satisfies $R(s)\geq -
\frac{1}{2}\;K(s)$ for certain
 function $K(s)$ that verifies
\begin{enumerate}
\item $K(s)\geq 0$,
 \item there exists a constant $a>0$ and a time $s_{1}> 0$,  which depend on $p$, $N$ and $\lambda$,
 such that \[ K(s)\leq
K(s_{1})\;e^{-(2\gamma+a)s},\quad\mbox{for}\quad s \geq
s_{1},\quad and\]  \item $\int^{\infty}_{0}e^{2\gamma s}K(s)\;ds
\leq C$, for a suitable positive constant $C$.
\end{enumerate}
\end{lem}
\begin{proof} We define the function $K(s)$ as follows
\[K(s)=2p\int_{{\mathbb{R}^{N}}}v^{p-1}v_{s}^{2}\rho\; dy.\]
 It is clear that $K(s)$ is nonnegative.  To see {\it 2}, we
 compute the derivative of the function $K(s)$ and write it conveniently.
\begin{eqnarray}\label{desigK}
\frac{dK(s)}{ds} &=&
2p\int_{{\mathbb{R}^{N}}}[(p-1)\;v^{p-2}v_{s}^{3}+2\;
v^{p-1}v_{s}v_{ss}]\rho\; dy\nonumber\\ &=&
\left(1+\frac{2}{p-1}\right)K(s)+2p\;(3p-1)\int_{{\mathbb{R}^{N}}}v^{2p-2}v_{s}^{2}\;\rho\;
dy\nonumber\\
& & -2p\;(p-1)\int_{{\mathbb{R}^{N}}}v^{p-2}v_{s}^{2}Lv\;\rho\;
dy-4p\int_{{\mathbb{R}^{N}}}v^{p-1}v_{s}Lv_{s}\;\rho\; dy.
\end{eqnarray}
 To get a bound of the second of the four terms of (\ref{desigK}), we use the inequality
(\ref{deswangv}) in order to obtain that
\begin{equation}\label{2term}
2p\;(3p-1)\int_{{\mathbb{R}^{N}}}v^{2p-2}v_{s}^{2}\;\rho\; dy\leq
\frac{3p-1}{(\lambda^{1-p}-1)(p-1)}\;\frac{e^{s}}{e^{s}-1}\;K(s).\end{equation}
For the bound of the third term, we apply first the lemmas
\ref{desiginteg1} and \ref{Gn} in order to deduce that
$vLv-\frac{N}{2}\;v^{2}\geq 0$ a.e., recalling that
$v^{p-3}v_{s}^{2}\geq 0$ we conclude
\begin{eqnarray*}
\int_{{\mathbb{R}^{N}}}(vLv-\frac{N}{2}v^{2})\;v^{p-3}v_{s}^{2}\;\rho\;
dy \geq  0\;.
\end{eqnarray*}This inequality quickly leads to a bound of the third term of
(\ref{desigK}) which is a multiple of $K(s)$.
\begin{eqnarray}\label{3term}
2p\;(p-1)\int_{{\mathbb{R}^{N}}}v^{p-2}v_{s}^{2}Lv\;\rho\; dy
&\geq & 2p\;(p-1)\;\frac{N}{2}
\int_{{\mathbb{R}^{N}}}v^{p-1}v_{s}^{2}\;\rho\; dy\nonumber\\
&=& (p-1)\;\frac{N}{2}\;K(s).\end{eqnarray}For the last term of
(\ref{desigK}), we use first that
$v_{s}Lv_{s}-\frac{N}{2}\;v^{2}_{s}\geq 0$ a.e. (it's owing to the
lemmas \ref{desiginteg1} and \ref{Gn}) and that also $v^{p-1}\geq
0$. We obtain
\begin{eqnarray*}
\int_{{\mathbb{R}^{N}}}(v_{s}Lv_{s}-v_{s}^{2})\;v^{p-1}\;\rho\;
dy\geq 0\;.
\end{eqnarray*}
This inequality permits to bound the last term of (\ref{desigK})
by a multiple of $K(s)$ as follows
\begin{eqnarray}\label{4term}
4p\int_{{\mathbb{R}^{N}}}v^{p-1}v_{s}Lv_{s}\;\rho\; dy &\geq&
2p\;N\int_{{\mathbb{R}^{N}}}v^{p-1}v_{s}^{2}\;\rho\;
dy\nonumber\\&=& N\; K(s).\end{eqnarray} Then, taking into account
 (\ref{2term}),(\ref{3term}) and (\ref{4term}), we have that
\begin{eqnarray*}\frac{dK(s)}{ds}\leq
&-&\left(N-\frac{2}{p-1}\right)K(s)+\left(1-(p-1)\;\frac{N}{2}\right)K(s)
\\&+&\frac{3p-1}{(\lambda^{1-p}-1)(p-1)}\;\frac{e^{s}}{e^{s}-1}\;K(s)\\\quad
 &=& -2\gamma \;K(s)-\mu(s)\;K(s).\end{eqnarray*}
where
\begin{eqnarray}\label{mu}\mu(s)&=&-1+(p-1)\;\frac{N}{2}-
\frac{(3p-1)}{(\lambda^{1-p}-1)(p-1)}\;\frac{e^{s}}{e^{s}-1}\nonumber\\
&=&(p-1)\;\gamma-\frac{(3p-1)}{(\lambda^{1-p}-1)(p-1)}\;f(s),\end{eqnarray}
with $f(s)=\frac{e^{s}}{e^{s}-1}$. Now, we must  prove that
$\mu(s)>0$ for $s\geq s_{1}$, where $s_{1}$ is a time which
depends on $p$, $N$ and $\lambda$. It is clear that, owing to
(\ref{mu}), it is equivalent to require that for $f(s)< B$, where
$B$ is the constant defined by
\[B=\frac{\gamma\;(p-1)^{2}\;(\lambda^{1-p}-1)}{3p-1}.\]
We can take $s_{1}=f^{-1}(\frac{1+B}{2})$ in order to get
$\mu(s)>0$ for $s\geq s_{1}$, since $B>1$ due to the hypotheses of
the lemma. Then, a bound of the derivative of $K(s)$ is obtained
from
\[\frac{dK(s)}{ds}\leq -(2\gamma + a)\; K(s),\quad s\geq s_{1},\]
where $a$ is a positive constant (we can take, for example,
$a=\frac{\mu(s_{1})}{2}$). Thus, from a certain $s_{1}> 0$, the
function $K(s)$ has exponential decay rate. That is,
\begin{equation}
K(s)\leq K(s_{1})\;e^{-(2\gamma+a)s}.
\end{equation}
This statement proves the second part of the lemma. The last part
 of the lemma can be deduced immediately
from the previous one.
\end{proof}

Applying the former lemma to (\ref{igdI}), we obtain the mentioned
decay for the entropy production as it will be proved in the next
theorem.

\begin{teo}\label{decaI}
Under the hypotheses of the lemma \ref{propk} and if, moreover,
$I(u_{0})<\infty$, then $I(v(s))$ has an exponential decay rate.
More precisely,
\[I(v(s))\leq \left[I(u_{0})+C\right]\;e^{-2\gamma s}.\]
\end{teo}
\begin{proof} From the expression (\ref{igdI}) and the lemma
\ref{propk}, we have
\[\frac {d}{ds}\left(e^{2\gamma s}\;I(v(s))\right) = -2\; e^{2\gamma s}\;R(s) \leq
e^{2\gamma s}\; K(s).\] \noindent Integrating between $0$ and $s$
we obtain
\begin{eqnarray*}
e^{2\gamma s}\;I(v(s)) -I(v(0)) & \leq & \int_{0}^{s}
e^{2\gamma\sigma}\;
K(\sigma)\; d\sigma\\
 & \leq & \int_{0}^{\infty}
e^{2\gamma\sigma}\; K(\sigma)\; d\sigma \leq C\\
\therefore \quad I(v(s)) & \leq & \left[I(u_{0})+C\right]\;
e^{-2\gamma s}.
\end{eqnarray*}
\end{proof}
Now, we are in a position to prove that the entropy functional
decays
 exponentially. It will be proved in the next theorem.
\begin{teo}\label{decaE} Under the hypotheses of the theorem
\ref{decaI}, the entropy functional $E(v(s))$ has exponential
decay, that is,\[E(v(s))\leq \ C\; e^{-2\gamma s},\quad\mbox{for
every}\;s\geq s_{1},\] for certain $s_{1}>0$ which depends on $p$,
$N$ and $\lambda$, and for certain positive constant $C$ that
depends on $I(u_{0})$, $p$, $N$ and $\lambda$.
  \end{teo}\begin{proof} From part {\it2} of proposition \ref{E} and  expression
  (\ref{igdI}),
  we have that
\begin{eqnarray}\label{desE} \frac{dE(v(s))}{ds}& =
&\frac{1}{2\gamma}\;\frac{dI(v(s))}{ds}+\frac{1}{\gamma}\;R\nonumber\\
& \geq &\frac{1}{2\gamma}\;\frac{dI(v(s))}{ds} -
\frac{1}{2\gamma}\;K(s),
\end{eqnarray}
where the inequality in (\ref{desE}) is owing to the bound of
$R(s)$ from $K(s)$. Now, we integrate between $s$ and $b$ in
(\ref{desE}) and we use part {\it 2} of lemma \ref{propk} to
obtain
\begin{eqnarray*} E(v(b))&-&E(v(s)) \geq\\
&\geq&\frac{1}{2\gamma}\;\left[I(v(b))-I(v(s))\right]-\frac{1}{2\gamma}\left.\left(\frac{
K(s_{1})}{2\gamma+a}e^{-(2\gamma+a)\sigma}\right)\right|^{b}_{s},
\end{eqnarray*} for every $s\geq s_{1}$. Taking the limit for
$b\rightarrow\infty$ and using part {\it 3} of proposition
\ref{E}, we get
\begin{eqnarray*}
E(v(s))&\leq & \frac{1}{2\gamma} \;I(v(s))+\frac{
K(s_{1})}{2\gamma\;(2\gamma+a)}\;e^{-(2\gamma+a)s},
\end{eqnarray*} for every $s\geq s_{1}$. Therefore, owing to the theorem (\ref{decaI}),
  the announced decay for the entropy functional takes place
 \[E(v(s))\leq \ C \;e^{-2\gamma s},\quad\mbox{for every}\;s\geq s_{1}.\]
\end{proof}

\section{Asymptotic behavior of the solution}

To finish with the application of this method, we must obtain a
bound of the norm of the solution $v(s)$ in terms of the entropy
$E(v(s))$ and the entropy production $I(v(s))$. For this purpose
we define the following function $g(s)$
\begin{equation}\label{defg}
 g(s)=\frac{1}{2}\int_{\mathbb{R}^{N}}v^{2}(s) \;\rho\;
 dy\;.\end{equation}
The next lemma provides a bound of the function $g(s)$ in terms of
$E(v(s))$ and $I(v(s))$, that is, a bound of the norm of $v(s)$ in
the space $L^{2}_{\rho}$.

\begin{teo}\label{desigg} Under the hypotheses of theorem (\ref{decaI}) and if,
 moreover, $p>\tilde p\;$, where
$\tilde p$ is defined by
\[\tilde p = \left\{
\begin{array}{l}
\frac{N}{N-2}\;,\quad N=3\;,\\1+\frac{4}{N}\;,\quad  N\geq 4\; ,
\end{array}
\right. \] then
\begin{enumerate}
\item $\left(\frac{p-1}{2}\;N-2\right)g(s)\leq\;
\frac{1}{2}\;I(v(s))+(p+1)\;E(v(s))$ and \item $g(s) \leq \; C
\;e^{-2\gamma s}\quad\mbox{for all}\;\;s\geq s_{1}$,
 where
$s_{1}$ is a positive number which depends on $p$, $N$ and
$\lambda$, and $C$ is certain positive constant which depends on
$I(u_{0})$, $p$, $N$ and $\lambda$.
\end{enumerate}
\end{teo}
\begin{proof}  We observe that the derivative of the function $g(s)$ satisfies
\begin{eqnarray}
g'(s) &=& \int_{\mathbb{R}^{N}}v v_{s}\; \rho\; dy
\label{derivg1}\\&\leq & \int_{\mathbb{R}^{N}}\left[
\frac{1}{2}\;v^{2}+\frac{1}{2}\;v_{s}^{2}\right]\rho\; dy =
g(s)+\frac{1}{2}\;I(v(s)).\label{derivg2}
\end{eqnarray}
On the other hand, we can obtain another expression for $g'(s)$
replacing  $v_{s}$ in (\ref{derivg1}) according to the problem
(\ref{ecv}) and using the definitions of $E(v(s))$, $L$ and
$g(s)$, that is,
\begin{equation}\label{derivg3}g'(s)=-(p+1)\;E(v(s))+\frac{p-1}{2}\;(Lv,v)-g(s).
 \end{equation}
Owing to (\ref{desigL}), (\ref{derivg2}) and (\ref{derivg3}), we
have that
\[\left(\frac{p-1}{2}\;N-2\right)g(s)\leq
\frac{1}{2}\;I(v(s))+(p+1)\;E(v(s))\;.\]  As $p>\frac{4}{N}+1$ and
using the bounds for the entropy and entropy production in
theorems \ref{decaI} and \ref{decaE}, respectively, we get
\[g(s)\leq C\; e^{-2\gamma s}\quad\mbox{para}\;\;s\geq s_{1}\;,\]
where $C$ is a positive constant which depends on $N$, $p$,
$\lambda$ and $I(v(0))$.
\end{proof}
Thanks to  the definition (\ref{defg}), we can  deduce the decay
of the norm of $v(s)$ in the space $L^{2}_{\rho}$ which is the aim
in the following theorem.
 \begin{teo} Under the hypotheses of the theorem \ref{desigg},
  then
   \[\|v(s)\|_{L^{2}_{\rho}}\leq C \;e^{-\gamma s}\;,\;\;s\geq s_{1}\;,\]
 for  $s_{1}>0$ which depends on $p$, $N$ and
 $\lambda$, and for certain positive constant $C$ which depends on $I(u_{0})$, $p$, $N$ and
 $\lambda$.
 \end{teo}
Now, we can use again the change of variables (\ref{cv}) in order
to obtain the decay of the solution $u(x,t)$ of (\ref{ec}),
getting the next result.
 \begin{teo}\label{deca-u} Under the hypotheses of the theorem \ref{desigg},
 then \[\|u(t)\|_{L^{2}}\leq C\;
(t+1)^{-\frac{N}{4}}\;,\;\; t\geq t_{1}\;,\]for $t_{1}>0$ which
depends on $p$, $N$ and
 $\lambda$, and for certain positive constant $C$ which depends on
 $I(u_{0})$, $p$, $N$ and
 $\lambda$.
 \end{teo}
\begin{rem}Let us observe that, because of the Wang's theorem and the theorem
\ref{deca-u}, we obtain the decay of the norm $\|u(t)\|_{L^{q}}$
for $q\geq 2$. This is,
\begin{eqnarray}\label{deca-u-q}
\|u(t)\|_{L^{q}}\sim
t^{-\frac{1}{p-1}-\frac{2}{q}\left(\frac{N}{4}-\frac{1}{p-1}\right)},
\end{eqnarray}
for $t\geq t_{1}$.
\end{rem}
\begin{rem}
Notice that, although the decay of the norm  $\|u\|_{L^{q}}\sim
t^{-\frac{N}{2}\left(1-\frac{1}{q}\right)}$ obtained by Kawanago
in \cite{Kawanago} is better than (\ref{deca-u-q}) in the case
$q>2$ and the same for $q=2$, the result obtained in
\cite{Kawanago} is true for $\frac{N+2}{N}<p<\frac{N+2}{N-2}$ and
the result obtained in this work corresponds to the range
$p>\tilde p\;$.
\end{rem}

\end{document}